\documentclass[reqno,11pt]{amsart}

\usepackage{amscd}
\usepackage{amsmath}
\usepackage{amssymb}
\usepackage[american]{babel}
\usepackage{bbm}
\usepackage{bookmark}
\usepackage{cmap} 
\usepackage{dsfont}
\usepackage{enumerate}
\usepackage{epigraph}
\usepackage[mathscr]{euscript}
\usepackage[myheadings]{fullpage}
\usepackage{graphicx}
\usepackage[geometry]{ifsym}
\usepackage{mathabx}
\usepackage{accents}
\usepackage{mathrsfs}
\usepackage{mathtools}
\usepackage{refcount}
\usepackage{setspace}
\usepackage{stmaryrd}
\usepackage{thinsp}
\usepackage{verbatim}
\usepackage[all,2cell]{xy} \UseTwocells
\usepackage{xr}
\usepackage[multiple]{footmisc}
\usepackage{hyperref}

\usepackage{tikz}
\usetikzlibrary{decorations.pathmorphing}
\usetikzlibrary{cd}

\newcommand{\e}{\mathds{1}}

\numberwithin{equation}{subsection}

\newtheorem{thm}{Theorem}[subsubsection]
\newtheorem*{thm*}{Theorem}

\newtheorem*{cor*}{Corollary}
\newtheorem{lem}[thm]{Lemma}

\newtheorem{prop}[thm]{Proposition}
\newtheorem{prop-const}[thm]{Proposition-Construction}

\newtheorem*{conjecture*}{Conjecture}

\newtheorem*{princ*}{Principle}

\theoremstyle{remark}
\newtheorem{rem}[thm]{Remark}

\newtheorem{defin}[thm]{Definition}

\newcounter{steps}[thm]
\newtheorem{step}[steps]{Step}

\newcommand{\xar}[1]{\xrightarrow{#1}}
\newcommand{\rar}[1]{\xar{#1}}
\newcommand{\isom}{\rar{\simeq}}

\newcommand{\into}{\hookrightarrow}

\newcommand{\Loc}{\on{Loc}}

\newcommand{\bB}{{\mathbb B}}

\newcommand{\bG}{{\mathbb G}}
\newcommand{\bH}{{\mathbb H}}

\newcommand{\bS}{{\mathbb S}}

\newcommand{\sC}{{\EuScript C}}
\newcommand{\sD}{{\EuScript D}}
\newcommand{\sE}{{\EuScript E}}
\newcommand{\sF}{{\EuScript F}}
\newcommand{\sG}{{\EuScript G}}
\newcommand{\sH}{{\EuScript H}}

\newcommand{\sM}{{\EuScript M}}

\newcommand{\sS}{{\EuScript S}}

\newcommand{\sV}{{\EuScript V}}
\newcommand{\sW}{{\EuScript W}}

\newcommand{\sY}{{\EuScript Y}}

\newcommand{\fA}{{\mathfrak A}}

\newcommand{\fL}{{\mathfrak L}}

\newcommand{\fT}{{\mathfrak T}}

\newcommand{\on}{\operatorname}

\newcommand{\mathendash}{\text{\textendash}}

\newcommand{\Ker}{\on{Ker}}

\newcommand{\Aut}{\on{Aut}}
\newcommand{\Spec}{\on{Spec}}

\newcommand{\id}{\on{id}}

\newcommand{\Rep}{\mathsf{Rep}}

\newcommand{\act}{\on{act}}
\newcommand{\actson}{\curvearrowright}
\newcommand{\coact}{\on{coact}}

\renewcommand{\dot}{\bullet}

\newcommand{\vph}{\varphi}

\newcommand{\Vect}{\mathsf{Vect}}

\newcommand{\Gr}{\on{Gr}}

\newcommand{\Whit}{\mathsf{Whit}}

\newcommand{\Bun}{\on{Bun}}

\newcommand{\LocSys}{\on{LocSys}}

\renewcommand{\mod}{\mathendash\mathsf{mod}}

\newcommand{\DGCat}{\mathsf{DGCat}}

\renewcommand{\lim}{\on{lim}}

\newcommand{\TwoHom}{\mathsf{Hom}}

\newcommand{\heart}{\heartsuit}

\newcommand{\Hecke}{\on{Hecke}}

\newcommand{\IndCoh}{\mathsf{IndCoh}}
\newcommand{\QCoh}{\mathsf{QCoh}}

\renewcommand{\subset}{\subseteq}

\newcommand{\AT}{\fA\fT}
\newcommand{\Ran}{\on{Ran}}
\newcommand{\Lisse}{\mathsf{Lisse}}

\newcommand{\Shv}{\mathsf{Shv}}
\newcommand{\temp}{\mathendash\on{temp}}
\newcommand{\Nilp}{\on{Nilp}}

\makeatletter
\newcommand{\biggg}{\bBigg@{4}}
\newcommand{\Biggg}{\bBigg@{5}}
\makeatother

\date{\today}

\begin{document}

\frenchspacing

\setlength{\epigraphwidth}{0.4\textwidth}
\renewcommand{\epigraphsize}{\footnotesize}

\begin{abstract}

Arinkin and Gaitsgory defined a category of \emph{tempered} $D$-modules
on $\Bun_G$ that is conjecturally equivalent to the category of
quasi-coherent (not ind-coherent!) sheaves on $\LocSys_{\check{G}}$.
However, their definition depends on the auxiliary data of a point of
the curve; they conjectured that their definition is independent of this
choice. Beraldo has outlined a proof of this conjecture that
depends on some technology that is not currently available. 
Here we provide a short, unconditional proof of the Arinkin-Gaitsgory conjecture.

\end{abstract}

\title{The Arinkin-Gaitsgory temperedness conjecture}

\author{Joakim F\ae rgeman}

\address{The University of Texas at Austin, 
Department of Mathematics, 
RLM 8.100, 2515 Speedway Stop C1200, 
Austin, TX 78712}

%\email{}

\author{Sam Raskin}

\address{The University of Texas at Austin, 
Department of Mathematics, 
RLM 8.100, 2515 Speedway Stop C1200, 
Austin, TX 78712}

\email{sraskin@math.utexas.edu}

\maketitle

\setcounter{tocdepth}{1}
\tableofcontents

\section{Introduction}

\subsection{Statement of the main theorem}

\subsubsection{}

Let $X$ be a geometrically connected, smooth, and projective curve over a field
$k$ of characteristic $0$. Let $G$ be a split reductive group over
$k$. Let $\Bun_G$ denote the moduli space of $G$-bundles on $X$, and
let $D(\Bun_G)$ denote the DG category of $D$-modules on $\Bun_G$.

Let $\check{G}$ denote the Langlands dual group to $G$, and let
$\LocSys_{\check{G}}$ denote the moduli space of $\check{G}$-bundles
on $X$ with connection.

\subsubsection{}

Let us begin by recalling some context from geometric Langlands.

Recall the geometric Langlands conjecture:
\begin{equation}\label{eq:gl}
D(\Bun_G) \simeq \IndCoh_{\on{Nilp}}(\LocSys_{\check{G}})
\end{equation}

\noindent which was given in this form by 
\cite{arinkin-gaitsgory}, following Beilinson-Drinfeld.

The right hand side has a subcategory $\QCoh(\LocSys_{\check{G}})$,
and the left hand side should have a parallel such subcategory.
Following \cite{arinkin-gaitsgory}, we refer to this putative 
subcategory of $D(\Bun_G)$ as the 
subcategory of \emph{tempered} $D$-modules on $\Bun_G$.

There are various (not obviously equivalent) proposals for the tempered subcategory.
One was given in \cite{arinkin-gaitsgory} \S 12, using
derived geometric Satake. It is dependent on a choice of point $x \in X(k)$;
we denote the resulting subcategory as $D(\Bun_G)^{x\temp}$. 
As in \cite{arinkin-gaitsgory}, a geometric Langlands equivalence
\eqref{eq:gl} that is equivalent with derived Satake at $x$ will necessarily
match $D(\Bun_G)^{x\temp}$ with $\QCoh(\LocSys_{\check{G}})$.

\subsubsection{}

We can now state our main theorem.

\begin{thm}\label{t:main}

The subcategory 
$D(\Bun_G)^{x\temp} \subset D(\Bun_G)$
is independent of the choice of point $x$.

\end{thm}

\noindent This result was proposed in \cite{arinkin-gaitsgory} Conjecture 12.7.5.

\subsection{Relation to work of Beraldo}

\subsubsection{}

A strategy of proof for Theorem \ref{t:main} was outlined by Dario Beraldo already
in 2015, yielding deeper results.
We describe the ingredients for his approach below.

\subsubsection{}

Roughly speaking, Beraldo's approach proceeds as follows. 

Beraldo has explained that 
a Ran space (or \emph{factorizable}) version of derived Satake would provide additional
symmetries of $D(\Bun_G)$, refining Gaitsgory's spectral action 
of $\QCoh(\LocSys_{\check{G}})$. Specifically, in \cite{dario-h},
has constructed a monoidal category $\bH(\LocSys_{\check{G}})$ receiving
a monoidal functor from $\QCoh(\LocSys_{\check{G}})$, and has conjectured
that the action of $\QCoh(\LocSys_{\check{G}})$ extends to 
$\bH(\LocSys_{\check{G}})$. He has observed that such an extension
would yield Theorem \ref{t:main}, and that such an extension should
follow from factorizable derived Satake 
(see \cite{dario-hochschild} \S 1.4.2 for related discussion, 
and \cite{dario-chiral} for a precise assertion in the Betti setting).

\subsubsection{}

Unfortunately, the factorizable derived Satake theorem has been 
slow to appear. It was claimed more than a decade ago by Gaitsgory-Lurie,
and again more recently by Justin Campbell and the second
author, where it is currently work in progress. In particular, at the time we 
are writing this, a definition of the spectral side has not yet appeared
publicly in written form. 
So the full derivation of the action of Beraldo's $\bH$ has remained 
somewhat heuristic. 

\subsubsection{}

Our purpose here is to provide a simple, unconditional proof of Theorem \ref{t:main}, 
sidestepping Beraldo's category $\bH$ and factorizable Satake. 

In particular, our argument does not resolve Beraldo's deep conjecture regarding
the action of $\bH$ on $D(\Bun_G)$. 
This remains an open problem, for which Beraldo's suggestion of
using factorizable Satake (once available) continues to appear to be the
most plausible strategy. Our work also does not settle 
other\footnote{See e.g. \cite{dario-oh-notes} for discussion 
of how an action of $\bH$ in the setting of \cite{agkrrv1} 
(and particularly \cite{agkrrv3}) would yield (arithmetic) Arthur parameters
for unramified automorphic representations.}
applications of Beraldo's conjecture.  

\subsection{Outline of the argument}

\subsubsection{}

The main ideas of our argument proceed as following. 

\subsubsection{}

For our point $x$, let $\sH_x^{sph}$ denote the associated (derived) spherical Hecke
category. There is a certain object $\AT_x \in \sH_x^{sph}$, which 
we call the \emph{anti-tempered unit} following \cite{dario-ramanujan}.

By definition, $D(\Bun_G)^{x\temp}$ is the kernel of the corresponding
Hecke functor:
\[
\AT_x \star -:D(\Bun_G) \to D(\Bun_G).
\]

\subsubsection{}

The point $x$ can be varied in the above description.

Specifically, there is a functor:
\[
\AT_X:D(\Bun_G) \to D(\Bun_G \times X)
\]

\noindent whose fiber at $x$ is the original functor
$\AT_x$, and similarly for any other point. 

\subsubsection{}

Roughly speaking, our idea is that (in a suitable sense)
the functor $D(\Bun_G) \to D(\Bun_G \times X)$ yields objects that
are locally constant along $X$, so the kernels of $\AT_x$ and $\AT_X$ coincide.

This is easier to explain in a slightly different context -- that
of sheaves with nilpotent singular support of \cite{agkrrv1}.
With notation as in \emph{loc. cit}., 
the corresponding Hecke functor:
\[
\AT_X:\Shv_{\Nilp}(\Bun_G) \to \Shv(\Bun_G \times X)
\]

\noindent maps into
$\Shv_{\Nilp}(\Bun_G) \otimes \mathsf{qLisse}(X)$ by universality of
the anti-tempered unit and 
the Nadler-Yun theorem
\cite{agkrrv1} Theorems 10.2.8 and 10.5.2 (which are  
following \cite{nadler-yun}). If e.g. we worked with complex curves, this would mean that
the functors $\AT_x$ and $\AT_y$ are the same up to choosing a path
between $x$ and $y$, and the Tannakian formalism applies 
in general.\footnote{In particular, this sketch provides a genuine
argument in the $\Shv_{\Nilp}$ setting, whether constructible
(as in \cite{agkrrv1}) or not (as in \cite{bz-n-betti-gl}, \cite{nadler-yun});
the Betti case may also be deduced directly from Beraldo's ideas 
via \cite{dario-chiral}. It should also be possible to adapt
\cite{dario-chiral} to the constructible \cite{agkrrv1} setting, but this
has not yet been done as far as we know.}

In the $D$-module setting, we use Gaitsgory's spectral action 
from \cite{generalized-vanishing} to essentially reduce to considering
Hecke eigensheaves, and then proceed from there. The reduction is in a similar spirit
to \cite{agkrrv1} \S 14.3-4. 

\begin{rem}

With that said, this note is logically independent
of \cite{agkrrv1}. Indeed, all of the ingredients in our argument were
already available when Arinkin-Gaitsgory formulated their conjecture.

\end{rem}

\subsection{Acknowledgements}

We thank Dima Arinkin, Dario Beraldo, and Dennis Gaitsgory for 
many productive conversations related to tempered $D$-modules.
The second author would also like to thank Dima Arinkin,
Dennis Gaitsgory, David Kazhdan, Nick Rozenblyum, and Yasha Varshavsky
for their collaboration on \cite{agkrrv1}, which was inspirational for
the present work. 

S.R. was supported by NSF grant DMS-2101984.

\section{Preliminary material}

Below, we collect some notation and basic constructions. 

We assume the reader is 
generally familiar with commonly used tools in de Rham geometric Langlands,
referring to \cite{dennis-laumonconf} for an introduction to 
these ideas. 

In what follows, $X$ is a geometrically connected,
smooth, projective curve over $k$. For $x\in X(k)$,
we let $i_x:\Spec(k) \to X$ denote the corresponding embedding.
We let $\Ran = \Ran_X$ denote the Ran space of $X$.

\subsection{Hecke functors}

We recall some preliminary constructions with Hecke functors parametrized
by points of $X$.

Below, we work over powers of the curve and Ran space.
For our point $x \in X(k)$, we let $\fL_x^+G$ (resp. $\fL G$, resp. $\Gr_{G,x}$)
denote the arc group (resp. loop group, resp. affine Grassmannian) 
based at this point. For a finite set $I$, 
let $\fL_{X^I}^+ G$ (resp. $\fL_{X^I} G$, resp. $\Gr_{G,X^I}$) denote the
standard corresponding space over $X^I$.

\subsubsection{}

For a finite set $I$, let $\sH_{X^I}^{sph} \coloneqq D(\Gr_{G,X^I})^{\fL_{X^I}^+G}$. 
Similarly, we let $\sH_{\Ran}^{sph}$ denote the Ran space version of the spherical
Hecke category, and $\sH_x^{sph}$ for the spherical category at a point $x$. 

We recall that $\sH_{\Ran}^{sph}$ is a monoidal DG category acting canonically on
$D(\Bun_G)$. We denote the product on $\sH_{\Ran}^{sph}$ and its
action on $D(\Bun_G)$ by $-\star-$.

\subsubsection{}

Let $\sF \in \sH_{X^I}^{sph}$ be given. 

On the one hand, $\sF$ defines an object of $\sH_{\Ran}^{sph}$, so 
a Hecke functor $\sF \star -:D(\Bun_G) \to D(\Bun_G)$.

There is also a closely related functor:
\[
\Hecke_{\sF}:D(\Bun_G) \to D(\Bun_G \times X^I) 
\]

\noindent constructed as follows. 
We have a standard
Hecke action functor:
\[
\sH_{X^I}^{sph} \otimes D(\Bun_G) \to D(\Bun_G).
\]

\noindent Considering the left hand side as a $(D(X^I),\overset{!}{\otimes})$-module
(via the action on the first functor), 
this action lifts uniquely:
\[
\xymatrix{
\sH_{X^I}^{sph} \otimes D(\Bun_G) \ar@{..>}[rr] \ar[drr] &&  
D(\Bun_G) \otimes D(X^I) \simeq D(\Bun_G \times X^I) \ar[d]^{\id \otimes C_{dR}^{\dot}(X^I,-)} \\
&& D(\Bun_G).
}
\]

\noindent of $D(X^I)$-module cateories. Finally, inserting $\sF$ on the first
tensor factor (in the dotted arrow above) 
gives the desired functor $\Hecke_{\sF}$.

We explicitly note that composing $\Hecke_{\sF}$ with de Rham cohomology
along $X^I$ gives $\sF \star -$.

\subsubsection{}

We remind the category $\Rep(\check{G})_{X^I}$ from \cite{cpsi} \S 6, and
the construction of the \emph{naive Satake functor}:
\[
\sS_{X^I}:\Rep(\check{G})_{X^I} \to \sH_{X^I}^{sph}.
\]

Similarly, we let:
\[
\sS_{\Ran}:\Rep(\check{G})_{\Ran} \to \sH_{\Ran}^{sph}
\]

\noindent denote the Ran space version, constructed out of the above functors.

\subsubsection{}

We will need the following technical notion in what follows. 

\begin{defin}

The subcategory $\sH_{X^I}^{sph,aULA} \subset \sH_{X^I}^{sph}$ 
of \emph{almost ULA} objects
the full (non-cocomplete) subcategory generated under finite colimits and direct summands
by applying $\sS_{X^I}$ to objects of $\Rep(\check{G})_{X^I}$ ULA over $X^I$.
The subcategory $\sH_{X^I}^{sph,qULA} \subset \sH_{X^I}^{sph}$ of
\emph{quasi-ULA} objects is the full subcategory generated under filtered colimits
by almost ULA objects. 

\end{defin}

\begin{rem}

We refer to \cite{cpsi} Appendix A and \S 6 for a convenient discussion of
ULA objects in this setting.

\end{rem}

\begin{rem}

Recall that e.g., the skyscraper sheaf $\delta_1 \in \sH_x^{sph}$ 
at the origin $1 \in \Gr_{G,x}$ is not compact; rather, it is \emph{almost compact}
in the technical sense. For similar reasons, 
the standard spherical sheaves over $X^I$ are not literally ULA over $X^I$;
we use the term \emph{almost ULA} in parallel with \emph{almost compact}.

\end{rem}

\subsection{Intermediate results}

We now formulate two intermediate results, 
from which we easily deduce Theorem \ref{t:main}.

\subsubsection{Local constancy}

Let $\sF \in \sH_X^{sph}$ be given. For $x \in X(k)$, let 
$\sF_x \in \sH_x^{sph}$ denote the $!$-fiber of $\sF$ at $x$.

We let: 
\[
\Hecke_{\sF}:D(\Bun_G) \to D(\Bun_G \times X)
\]

\noindent denote the following functor. 

By construction, the composition:
\[
D(\Bun_G) \xar{\Hecke_{\sF}} D(\Bun_G \times X)
\xar{(\id \times i_x)^!)}
D(\Bun_G)
\]

\noindent is the usual Hecke functor: 
\[
\sF_x \star -: D(\Bun_G) \to D(\Bun_G) 
\]

\noindent defined by $\sF_x$.

\subsubsection{}

With the above preliminary constructions out of the way, we can state:

\begin{thm}\label{t:lisse-ker}

Suppose $\sF \in \sH_X^{sph}$ is quasi-ULA. 
Then $\Ker(\Hecke_{\sF}) = \Ker(\sF_x \star -)$.

\end{thm}

This is the main technical result of the present paper; its proof 
is given in \S \ref{s:pf}.

\subsubsection{Projectors}

We follow terminology from \cite{dario-ramanujan}.

Define the \emph{tempered unit (at $x$)} 
$\e_x^{\tau} \in \sH_x^{sph}$ as follows. We recall the
\emph{derived Satake theorem} of \cite{bezrukavnikov-finkelberg}, which asserts:
\[
D(\Gr_{G,x})^{\fL_x^+G} \simeq \IndCoh_{\on{Nilp}}((\bB \check{G})^{\bS^2}) 
\subset \IndCoh((\bB \check{G})^{\bS^2}).
\]

\noindent There are adjoint functors:
\[
\Xi:\QCoh(\bB \check{G})^{\bS^2}) \rightleftarrows 
\IndCoh((\bB \check{G})^{\bS^2}):\Psi.
\]

\noindent Moreover, the unit object in $\sH_x^{sph}$ corresponds
to the trivial representation $\mathsf{triv} \in \Rep(\check{G})^{\heart} = 
\IndCoh_{\on{Nilp}}((\bB \check{G})^{\bS^2})^{\heart}$. 
We then take $\e_x^{\tau}$ to correspond to $\Xi \Psi(\mathsf{triv})$.

\subsubsection{}

By definition, there is a canonical map:
\[
\e_x^{\tau} \to \delta_1 \in \sH_x^{sph}.
\]

\noindent We then define the \emph{anti-tempered unit (at $x$)} as:
\[
\AT_x \coloneqq \Ker(\e_x^{\tau} \to \delta_1).
\]

By definition, an object $\sG \in D(\Bun_G)$ lies in 
$D(\Bun_G)^{x\temp}$ if and only if $\AT_x \star \sG = 0$.

\subsubsection{}

We now have the following basic observation.

\begin{lem}\label{l:at-qula}

There is a canonical object $\AT \in \sH_X^{sph,qULA}$ (not depending
on the choice of point $x \in X(k)$) with $!$-fiber
$\AT_x \in \sH_x^{sph}$ at $x$.

\end{lem}

\begin{proof}

This essentially follows from the universality of the construction of $\AT_x$.
We include more details below.

Let $\widehat{\sD}$ be \emph{some} formal disc.
Let $\Aut$ denote the group indscheme of its automorphisms. Let
$\Aut^{\star} \subset \Aut$ denote the group subscheme of automorphisms
fixing the closed point of $\widehat{\sD}$; we remind that
$\Aut^{\star} \to \Aut$ is an isomorphism modulo nilpotent ideals.
The group $\Aut$ acts strongly on
$\sH_x^{sph}$.  

By a standard construction, any $\Aut$-equivariant object $\sF_0$ of $\sH^{sph}$ 
(the spherical category corresponding to $\widehat{\sD}$)
gives rise to an object $\sF \in \sH_X^{sph}$. We claim any
resulting such objects are quasi-ULA; indeed, $(\sH^{sph})^{\Aut}$ 
is generated under colimits by objects in the heart of its $t$-structure,
and the heart of its $t$-structure is exactly $\Rep(\check{G})^{\heart}$,
and these objects map to the standard (almost ULA) objects of
$\sH_X^{sph}$ (cf. \cite{central} Proposition 1).

Next, we observe that we have a projection $\pi:\Aut^{\star} \to \bG_m$ 
with pro-unipotent kernel. Moreover, every object of
$\sH^{sph}$ is automatically equivariant with respect
to the kernel $\pi$; indeed, by pro-unipotence, this can be checked
on generators, and then it follows in the previous paragraph.
Moreover, this same logic shows every object is $\Aut$-monodromic,
or equivalently (after a choice of coordinate), 
$\bG_m$-monodromic for $\bG_m$ acting by loop rotation.

We now observe that $\sH^{sph}$ carries a canonical endofunctor
corresponding to $\Xi\Psi$ on the spectral side. One readily checks
that $\Xi\Psi$ is (canonically) strongly $\Aut$-equivariant by using
\cite{bezrukavnikov-finkelberg},
noting that their form of derived Satake describes the loop equivariant category,
so can be understood to be $\Aut$-equivariant in a suitable sense by the above. 
This concludes the argument.

\end{proof}

We now observe that Theorem \ref{t:main} follows immediately from 
Lemma \ref{l:at-qula} and Theorem \ref{t:lisse-ker}.

\begin{rem}

To avoid the subtleties involved in the above argument,
one could also proceed as follows. First, 
by \cite{dario-ramanujan} Theorem 1.4.8, $\Ker(\AT_x\star -) = \Ker(\sW\sS_0 \star -)$
for $\sW\sS_0$ as in \emph{loc. cit}., i.e., one takes the 
unit spherical Whittaker sheaf in $\Whit_x^{sph} \coloneqq 
D(\Gr_{G,x})^{\fL_x N,\psi}$ and $*$-averages it to $D(\Gr_G)^{\fL_x^+G}$.
This description of $\sW\sS_0$ manifestly extends to define a quasi-ULA 
(even almost ULA) object
$\sW\sS_{0,X} \in \sH_X^{sph}$, to which we could then apply
Theorem \ref{t:lisse-ker}.

\end{rem}

\subsection{Gaitsgory's spectral action}

We now review the main results of \cite{generalized-vanishing};
see also \cite{dennis-laumonconf} \S 4.3-4.5 and \S 11.1.

First, there is a canonical symmetric monoidal functor:
\[
\Loc:\Rep(\check{G})_{\Ran} \to \QCoh(\LocSys_{\check{G}})
\]

\noindent from \emph{loc. cit}. It admits a fully faithful continuous right
adjoint (cf. \emph{loc. cit}.); therefore, the restriction functor:
\[
\QCoh(\LocSys_{\check{G}})\mod \to \Rep(\check{G})_{\Ran}\mod
\]

\noindent is fully faithful. (Here modules are taken in the
symmetric monoidal category $\DGCat_{cont}$ of cocomplete DG categories).

On the other hand, there is an action of $\Rep(\check{G})_{\Ran}$ on $D(\Bun_G)$
that is constructed as:
\[
\Rep(\check{G})_{\Ran} \xar{\sS_{\Ran}} \sH_{\Ran}^{sph} \actson D(\Bun_G).
\]

\begin{thm}[Gaitsgory, \cite{generalized-vanishing}, 
\cite{dennis-laumonconf} Theorem 4.5.2] 
\label{t:vanishing}

The above action of $\Rep(\check{G})_{\Ran}$ on $D(\Bun_G)$ factors through
a (necessarily unique) action of $\QCoh(\LocSys_{\check{G}})$ via the 
localization functors.

\end{thm}

\begin{rem}

Related results in other contexts have also recently been obtained:
see \cite{nadler-yun}, \cite{agkrrv1}, \cite{fargues-scholze}. In these other
contexts, the proofs are more conceptual. 

\end{rem}

We again use $-\star-$ to denote the action of $\QCoh(\LocSys_{\check{G}})$
on $D(\Bun_G)$.

\section{Proof of Theorem \ref{t:lisse-ker}}\label{s:pf}

As above, the proof of Theorem \ref{t:main} reduces to Theorem \ref{t:lisse-ker}.
The purpose of this section is to prove the latter result.

\subsection{Setup}

It is clear that $\Ker(\Hecke_{\sF}) \subset \Ker(\sF_x \star -)$. 
So it remains to show the converse. We therefore fix 
$\sG \in \Ker(\sF_x \star -) \subset D(\Bun_G)$ and aim 
to show that $\sG \in \Ker(\Hecke_{\sF})$.

\subsubsection{}

We have an action functor:
\[
\act:\QCoh(\LocSys_{\check{G}}) \otimes D(\Bun_G) \to D(\Bun_G).
\]

\noindent As the first factor is canonically self-dual, we obtain
a functor:
\[
\coact:D(\Bun_G) \to \QCoh(\LocSys_{\check{G}}) \otimes D(\Bun_G).
\]

\subsubsection{}\label{sss:strategy}

We now form the following commutative diagram, whose analysis 
is central to the argument.
\[
\begin{tikzcd}
D(\Bun_G)
\arrow[rr,"\coact"]
\arrow[ddrr, bend right = 5,"\Hecke_{\sF}"' near end] 
\arrow[ddrrrr, bend right = 40, "\sF_x \star -"']
& & 
D(\Bun_G) \otimes \QCoh(\LocSys_{\check{G}})
\arrow[d,"\Hecke_{\sF} \otimes \id"] 
\\ & & 
D(\Bun_G \times X) \otimes \QCoh(\LocSys_{\check{G}}) 
\arrow[rr,"(\id \times i_x)^! \otimes \id"] 
\arrow[d,"{\id \otimes \Gamma(\LocSys_{\check{G}},-)}"]
& &
D(\Bun_G) \otimes \QCoh(\LocSys_{\check{G}}) 
\arrow[d,"{\id \otimes \Gamma(\LocSys_{\check{G}},-)}"] 
\\ & & 
D(\Bun_G \times X) 
\arrow[rr,"(\id \times i_x)^!"] 
& & 
D(\Bun_G)
\end{tikzcd}
\]

We consider $\sG$ as an object of the top left term. By assumption, 
it is mapped to $0$ in the bottom right term. Our goal is to show that
it maps to zero in the bottom left term. We will do so by showing the
following:

\begin{itemize}

\item (\S \ref{ss:st-1}) $\sG$ maps to zero in the top term of the rightmost column of the diagram, i.e.:
\begin{equation}\label{eq:st-1}
((\id \times i_x^!) \otimes \id)(\Hecke_{\sF}\otimes \id)\coact(\sG) = 0.
\end{equation}

\item (\S \ref{ss:st-2}) $\sG$ maps to zero in the middle term of the second column of the diagram, i.e.: 
\begin{equation}\label{eq:st-2}
(\Hecke_{\sF}\otimes \id)\coact(\sG) = 0.
\end{equation}

\end{itemize} 

\noindent Clearly the latter claim suffices.

\subsection{Step 1}\label{ss:st-1}

We begin by establishing \eqref{eq:st-1}.

\subsubsection{Reduction}

We have the following standard observation.

\begin{lem}\label{l:lisse-cons}

Suppose that $\sY$ is a QCA algebraic stack in the sense of \cite{dg-finiteness} and
suppose that $\sC$ is a DG category. Then an object:
\[
\sF \in \sC \otimes \QCoh(\sY)
\]

\noindent is zero if and only if for every $\sE \in \QCoh(\sY)$, we 
have:
\[
\big(\id \otimes \Gamma(\sY,-)\big)(\sF \otimes \sE) = 0 \in \sC.
\]

\noindent Here we consider $\sC \otimes \QCoh(\sY)$ as a module
category for $\QCoh(\sY)$ in the evident way, writing the action the right.

\end{lem}

\begin{proof}

More generally, for a dualizable DG category $\sD$, and object:
\[
\sF \in \sC \otimes \sD
\]

\noindent is zero if and only if $(\id \otimes \lambda)(\sF) = 0 \in \sC$
for every $\lambda \in \sD^{\vee}$, as a functor
$\Vect \xar{\sF} \sC \otimes \sD$ is equivalent by duality to a functor
$\sD^{\vee} \to \sC$. Now the claim follows from the existence of
perfect self-duality for QCA stacks, cf. \cite{dg-finiteness} \S 4.3.7.

\end{proof}

Therefore, it suffices to show that for any $\sE \in \QCoh(\LocSys_{\check{G}})$,
we have:
\begin{equation}\label{eq:w/e}
(\id \otimes \Gamma(\LocSys_{\check{G}},-))
\Bigg(\Big(((\id \times i_x^!) \otimes \id)(\Hecke_{\sF}\otimes \id)\coact(\sG) \Big) \otimes
\sE\Bigg) = 0.
\end{equation}

\subsubsection{}

We now manipulate the left hand side of \eqref{eq:w/e}.

We have:
\[
\begin{gathered}
(\id \otimes \Gamma(\LocSys_{\check{G}},-))
\Bigg(\Big(((\id \times i_x^!) \otimes \id)(\Hecke_{\sF}\otimes \id)\coact(\sG) \Big) \otimes
\sE\Bigg) = \\
(\id \otimes \Gamma(\LocSys_{\check{G}},-))
((\id \times i_x^!) \otimes \id)
(\Hecke_{\sF}\otimes \id)\Big(\coact(\sG) \otimes \sE\Big).
\end{gathered}
\]

We now observe that $\coact$ is a morphism of $\QCoh(\LocSys_{\check{G}})$-bimodules,
considering $D(\Bun_G)$ as a bimodule via the spectral action and symmetric
monoidality of $\QCoh(\LocSys_{\check{G}})$.
Therefore, we can rewrite the above as:
\[
(\id \otimes \Gamma(\LocSys_{\check{G}},-))
((\id \times i_x^!) \otimes \id)
(\Hecke_{\sF}\otimes \id)\coact(\sE \star \sG).
\]

\noindent By the big diagram of \S \ref{sss:strategy}, this term coincides with:
\[
\sF_x \star (\sE \star \sG).
\]

\noindent Therefore, it suffices to show that this term vanishes.

\subsubsection{}\label{sss:ker-submod}

By the above, it remains to show that 
$\Ker(\sF_x \star -) \subset D(\Bun_G)$ is 
a $\QCoh(\LocSys_{\check{G}})$-submodule category. 
Reformulating this using Theorem \ref{t:vanishing},
it suffices to show that it is a $\Rep(\check{G})_{\Ran}$-submodule category. 
I.e., we wish to show that for any $\sV \in \Rep(\check{G})_{\Ran}$,
$\sF_x \star \sV \star \sG = 0$.

As $\Rep(\check{G})_{\Ran}$ is generated as a monoidal category by
its subcategory $\Rep(\check{G})_X$, we can assume $\sV$ lies
in this subcategory. By excision, we can treat separately the
cases where $\sV$ is $*$-extended from $\Rep(\check{G})_{X\setminus x}$
and $\Rep(\check{G})_x$. In the former case, it follows as 
$\sV$ commutes with $\sF_x$ (Hecke functors at different points obviously
commute). In the latter case, it follows as $\sV$ commutes with
$\sF_x$, e.g., by the existence of the 
\emph{pointwise} symmetric monoidal structure on the
derived Satake category established in \cite{bezrukavnikov-finkelberg}.

This concludes the proof of \eqref{eq:st-1}.

\subsection{Step 2}\label{ss:st-2}

We now prove \eqref{eq:st-2}. This requires some digressions.

\subsubsection{Lisse sheaves}

Suppose $\sY$ is an Artin stack. 

We define 
$\Lisse_{\sY}(X) \subset \QCoh(\sY) \otimes D(X)$ to be the full
DG subcategory generated by under colimits by (finite rank)
vector bundles on $\sY \times X_{dR}$. We consider objects
of $\Lisse_{\sY}(X)$ as \emph{$\sY$-families of lisse $D$-modules on $X$}.

Let $x \in X(k)$.
We abuse notation in letting $i_x^!$ denote the composition:
\[
\Lisse_{\sY}(X) \into \QCoh(\sY) \otimes D(X) \xar{\id \otimes i_x^!} \QCoh(\sY) \otimes 
\Vect =
\QCoh(\sY).
\]

We will use the following result.

\begin{prop}

Suppose $\sY$ is locally almost of finite type and eventually coconnective.
Then the functor $i_x^!$ is conservative. 

More generally, for any dualizable DG category $\sC$, the functor:
\[
\id_{\sC} \otimes i_x^!: \sC \otimes \Lisse_{\sY}(X) \to \sC \otimes \QCoh(\sY)
\]

\noindent is conservative.

\end{prop}

\begin{proof}

\step 

First, we note that if $F:\sD_1 \to \sD_2 \in \DGCat_{cont}$ is 
conservative and $\sC \in \DGCat_{cont}$ is dualizable, then
$\id_{\sC} \otimes F:\sC \otimes \sD_1 \to \sC \otimes \sD_2$ is conservative.
Indeed, we can rewrite this functor as:
\[
\sC \otimes \sD_1 = 
\TwoHom_{\DGCat_{cont}}(\sC^{\vee},\sD_1) \xar{\vph \mapsto F\vph}
\TwoHom_{\DGCat_{cont}}(\sC^{\vee},\sD_2) = 
\sC \otimes \sD_2
\]

\noindent in which form it is manifestly conservative. 
Therefore, we are reduced to considering $\sC = \Vect$ in the assertion.

\step 

Next, suppose $S$ 
is an eventually coconnective scheme locally almost of finite type.
Let $|S|$ denote the set of points of its underlying topological space;
for $s \in |S|$, we write $\kappa(s)$ for the residue field at this point,
$s$ for $\Spec(\kappa(s))$,
and $i_s:s \to S$ for the structural morphism.

We then note that the functor: 
\[
\QCoh(S) \xar{\{i_s^*\}_{s \in |S|}} \prod_{s \in |S|} \QCoh(s)
\]

\noindent is conservative. Indeed, this follows from \cite{sag} Lemma 2.6.1.3
and the conservativeness of the restriction $S^{cl} \into S$ (which
is easy from eventual coconnectivity of $S$).

In our setting, let $\pi:S \to \sY$ be a flat cover. 
We find that the restriction functor:
\[
\QCoh(\sY) \to \prod_{s \in |S|} \QCoh(s)
\]

\noindent is conservative. 
By the same reasoning as before, for any dualizable DG category $\sD$, 
the functor:
\[
\sD \otimes \QCoh(\sY) \to \sD \otimes \prod_{s \in |S|} \QCoh(s) \isom 
\prod_{s \in |S|} \sD \otimes \QCoh(s)
\]

\noindent is conservative. In particular, this applies for $\sD = D(X)$.

\step 

By the above, we have a commutative diagram:
\[
\begin{tikzcd}
\Lisse_{\sY}(X) \arrow[r,hookrightarrow] \ar[d] 
& 
D(X) \otimes \QCoh(\sY) 
\arrow[d] 
\arrow[rr,"i_x^! \otimes \id"] 
&& \QCoh(\sY) \ar[d]
\\
\prod_{s \in |S|} \Lisse_s(X) \arrow[r,hookrightarrow]
&
\prod_{s \in |S|} D(X) \otimes \QCoh(s) 
\ar[rr,"i_x^! \otimes \id"] && \prod_{s \in |S|} \QCoh(s).
\end{tikzcd}
\]

\noindent The middle and right vertical arrows are conservative, 
so the same is true of the left vertical arrow. Therefore, to see
that the top line is conservative, it suffices to show that 
for each $s \in |S|$, the functor:
\[
i_x^!:\Lisse_s(X) \to \QCoh(s)
\]

\noindent is conservative.

Therefore, we are reduced to the case where $S = \Spec(\kappa)$
for some field $\kappa/k$.

\step 

Let $X_\kappa \coloneqq X \times_{\Spec(k)} \Spec(\kappa)$. Note
that $D(X) \otimes \Vect_{\kappa} = D_{/\kappa}(X_\kappa)$, where we
regard $X_\kappa$ as a scheme over the field $\kappa)$
and write $D_{/\kappa}$ to emphasize this (reminding that
implicitly, the category of $D$-modules depends on the structural
map to $\Spec$ of a field).
Moreover, $X_{dR} \times \Spec(\kappa) = X_{\kappa,dR/\Spec(\kappa)}$,
so $\Lisse_{\Spec(\kappa)}(X) \subset D_{/\kappa}(X_\kappa)$ 
is the subcategory of ($\Spec(\kappa)$-families of) 
lisse $D$-modules on $X_{\kappa}$, considering the latter as 
a scheme over $\Spec(\kappa)$. 

This is all to say that we are reduced to the case where $\kappa = k$,
as the only difference is notational.

\step 

We are now essentially done: the functor 
$i_x^!:\Lisse(X) \coloneqq \Lisse_{\Spec(k)}(X) \to \Vect$
is $t$-exact up to shift and is obviously conservative on the heart
of the $t$-structure, so is conservative (as the $t$-structure on
$\Lisse(X)$ is left separated).

\end{proof}

\subsubsection{}

We now observe the following.

\begin{lem}\label{l:qula-lisse}

For any quasi-ULA $\sF$, the composition:
\[
\begin{gathered}
D(\Bun_G) \xar{\coact} D(\Bun_G) \otimes \QCoh(\LocSys_{\check{G}}) 
\xar{\Hecke_{\sF} \otimes \id} D(\Bun_G \times X) \otimes \QCoh(\LocSys_{\check{G}})
= \\ D(\Bun_G) \otimes D(X) \otimes \QCoh(\LocSys_{\check{G}})
\end{gathered}
\]

\noindent maps into the subcategory:
\[
D(\Bun_G) \otimes \Lisse_{\LocSys_{\check{G}}}(X).
\]

\end{lem}

\begin{proof}

First, note that:
\[
D(\Bun_G) \otimes \Lisse_{\LocSys_{\check{G}}}(X) \to 
D(\Bun_G) \otimes D(X) \otimes \QCoh(\LocSys_{\check{G}})
\]

\noindent is indeed fully faithful: e.g., the embedding
$\Lisse_{\LocSys_{\check{G}}}(X) \into 
D(X) \otimes \QCoh(\LocSys_{\check{G}})$ admits a continuous right adjoint
by definition, so tensoring with it preserves fully faithfulness.

Now, by definition of quasi-ULAness, we are immediately reduced to 
considering the case where $\sF$ is almost ULA. Such an object is cohomologically
bounded, so we are reduced to the case where $\sF$ is concentrated
in degree zero. 

In this case, $\sF$ necessarily is a direct sum of 
terms of the form $\sS_X(V \otimes \sigma)$ where
$\sigma \in D(X)^{\heart}$ is a finite rank local system, 
$V \in \Rep(\check{G})^{\heart}$ is finite dimensional, 
we consider $V \otimes \sigma$ as an object of $\Rep(\check{G})_X$, 
and we remind that $\sS_X$ denotes the geometric Satake functor
(cf. \cite{cpsi} \S 6, especially Proposition 6.22.1 and Lemma 6.23.1).

Now observe that $\Hecke_{\sS_X(V \otimes \sigma)}$ differs from 
$\Hecke_{\sS_X(V \otimes \omega_X)}$ by applying
$\id_{D(\Bun_G)} \otimes (\sigma \overset{!}{\otimes} -) \otimes 
\id_{\QCoh(\LocSys_{\check{G}})}$. Clearly this operation
preserves the subcategory 
$D(\Bun_G) \otimes \Lisse_{\LocSys_{\check{G}}}(X)$, so 
we may take $\sF = \sS_X(V \otimes \omega_X)$ instead.  We simplify the
notation by writing $\sF = \sS_X(V)$.

Next, recall that $V \in \Rep(\check{G})$ defines a canonical
vector bundle $\sE_V$ on $X_{dR} \otimes \LocSys_{\check{G}}$.
We then observe that the compositions:
\[
D(\Bun_G) \xar{\coact} D(\Bun_G) \otimes \QCoh(\LocSys_{\check{G}}) 
\xar{\Hecke_{\sS(V)} \otimes \id} 
D(\Bun_G) \otimes D(X) \otimes \QCoh(\LocSys_{\check{G}})
\]

\noindent and: 
\[
\begin{gathered}
D(\Bun_G) \xar{\coact} D(\Bun_G) \otimes \QCoh(\LocSys_{\check{G}}) 
\xar{\id \otimes \sE_V \otimes \id} \\
D(\Bun_G) \otimes D(X) \otimes \QCoh(\LocSys_{\check{G}})
\otimes \QCoh(\LocSys_{\check{G}}) 
\xar{\id \otimes \id \otimes (- \otimes -)} \\
D(\Bun_G) \otimes D(X) \otimes \QCoh(\LocSys_{\check{G}})
\end{gathered}
\]

\noindent coincide (by construction\footnote{Specifically, 
we use the following fact, which is tautological from the construction of $\Loc$.
Suppose $\sM \in D(X)$. We obtain an object $V \otimes \sM \in \Rep(\check{G})_X$.
Let $\lambda_{\sM}:D(X) \to \Vect$ be the functor Verdier dual to $\sM$,
i.e., the functor $C_{dR}^{\dot}(X,\sM\overset{!}{\otimes} -)$.  
Then $\Loc(V \otimes \sM) \in \QCoh(\LocSys_{\check{G}})$ 
is (functorially in $\sM$) calculated as the image of
$\sE_V$ under $\lambda_{\sM} \otimes \id:D(X) \otimes \QCoh(\LocSys_{\check{G}})
\to \QCoh(\LocSys_{\check{G}})$.}
of $\Loc$). The latter clearly maps into 
$D(\Bun_G) \otimes 
\Lisse_{\LocSys_{\check{G}}}(X)$, as desired.

\end{proof}

\subsubsection{}

By Lemma \ref{l:qula-lisse}, we have:
\[
(\Hecke_{\sF}\otimes \id)\coact(\sG) \in 
D(\Bun_G) \otimes \Lisse_{\LocSys_{\check{G}}}(X).
\]

\noindent Moreover, by \eqref{eq:st-1}, this object
vanishes when we apply $(\id \otimes i_x^!)$ to it.
Therefore, by Lemma \ref{l:lisse-cons}, we have:
\[
(\Hecke_{\sF}\otimes \id)\coact(\sG) = 0.
\]

\noindent Here we observe that $D(\Bun_G)$ is dualizable by
\cite{bung-cpt}, and that $\LocSys_{\check{G}}$ is eventually
coconnective e.g. by \cite{arinkin-gaitsgory} \S 10.
This concludes the proof of \eqref{eq:st-2}, hence of
Theorem \ref{t:lisse-ker}.

\bibliography{bibtex.bib}{}
\bibliographystyle{alphanum}

\end{document}